\documentclass[10pt]{article}

\usepackage{amsfonts,amssymb}

\input{xy}
\usepackage[all]{xy}
\xyoption{all}
\xyoption{poly}

\newtheorem{thm}{Theorem}[section]
\newtheorem{cor}[thm]{Corollary}
\newtheorem{lem}[thm]{Lemma}
\newtheorem{pro}[thm]{Proposition}
\newtheorem{defi}[thm]{Definition}
\newtheorem{rem}[thm]{Remark}

\newenvironment{proof}{\noindent {\textbf{Proof.}} \sf}

\def\qed{\hfill $\diamond$ \bigskip}

\def\B{{\mathcal B}}
\def\C{{\mathcal C}}
\def\D{{\mathcal D}}

\def\F{{\mathcal F}}

\def\lim{\mathop{\rm lim}\nolimits}

\def\aut{\mathsf{Aut}}

\def\st{\mathsf{St}}
\def\mor{\mathsf{Mor}}
\def\gal{\mathsf{Gal}}

\begin{document}

\sf

\title{Full and convex linear subcategories are incompressible}
\author{Claude Cibils, Maria Julia Redondo and Andrea Solotar
\thanks{\footnotesize This work has been supported by the projects  UBACYTX212 and 475, PIP-CONICET 112- 200801-00487, PICT-2007-02182 and MATHAMSUD-NOCOMALRET.
The second and third authors are  research members of
CONICET (Argentina).}}

\date{}

\maketitle

\begin{abstract}
Consider the  intrinsic fundamental group \textit{\`{a} la} Grothendieck of a linear category, introduced in \cite{CRS} and \cite{CRS2} using connected gradings. In this article we prove that any full convex subcategory is incompressible, in the sense that the group map between the corresponding  fundamental groups is injective.  We start by proving the functoriality of the  intrinsic fundamental group with respect to full subcategories, based on the study of the restriction of connected gradings.
\end{abstract}


\section{\sf Introduction}

In two recent papers \cite{CRS,CRS2} we have considered a new intrinsic fundamental group attached to a linear category. In \cite{CRS2} we have made several explicit computations: for instance the fundamental group of $kC_p$ where $C_p$ is a group of prime order $p$ and $k$ is a field of characteristic $p$ is the direct product of $C_p$ with the infinite cyclic group.

We have obtained this group using methods inspired by the definition of the fundamental group in other mathematical contexts. We briefly recall the definition of the intrinsic fundamental group as the automorphism group of a fibre functor.

Previously a non canonical fundamental group has been introduced by R.
Mar\-t\'i\-nez-Villa and J.A. de la Pe\~{n}a in \cite{MP} and K. Bongartz and P. Gabriel in \cite{boga} and \cite{ga}. This group  varies considerably according to the presentation of the linear category, see for instance \cite{asde,buca}. See also  \cite{le,le2} as a first approach to solve this variability problem.

The main tool we use are connected gradings of linear categories. The fundamental group that we consider is a group which is derived from all the groups grading the linear category in a connected way. More precisely each connected grading provides a Galois covering through the smash product, see \cite{CM}. Considering the category of Galois coverings of this type we define the intrinsic fundamental group as the automorphism group of the fibre functor over a chosen object.

One of the main purposes of this paper is to prove that full and convex subcategories are incompressible, in the sense used in algebraic topology where a subspace is called incompressible if the group map between the corresponding fundamental groups is injective.

Another main purpose is to prove that the intrinsic fundamental group is functorial with respect to full subcategories, answering in this way a question by Alain Brugui\`{e}res. Note that this is not automatic from the definition of the fundamental group. We provide a description of elements of the intrinsic fundamental group as compatible families of elements lying in each group which grades the linear category in a connected way. The other main ingredient for proving the functoriality is the connected component of the base object for a restricted grading, and the associated connected grading considered in \cite{CM} depending on some choices which we prove to be irrelevant at the intrinsic fundamental group level.

Finally we consider convex subcategories. Recall that a linear subcategory of a linear category is convex if morphisms of the subcategory can only be factorized through morphisms in the subcategory. Our results enables to prove that full and convex subcategories are incompressible.

We thank the referee for useful comments.

\section{\sf Coverings and elements of the intrinsic fundamental group}

In this section we recall some definitions and results from \cite{CRS} and \cite{CRS2} that we will use throughout this paper.
A main purpose of this article is to prove that the intrinsic fundamental group of a $k$-category is functorial with respect to
full subcategories. In order to do so we will need a concrete interpretation of elements of this group that we provide below.

Let $k$ be a commutative ring. A $k$-category is a small category
$\B$ with set of objects $B_0$ such that each morphism set ${}_y\B_x$ from
$x\in\B_0 $ to $y\in\B_0$ is a $k$-module, the composition of
morphisms is $k$-bilinear and the identity at each object is central
in its endomorphism ring.

\begin{defi} The star $\st_{b}\B$ of a $k$-category $\B$ at an
object $b$ is the direct sum of all $k$-modules of morphisms with source or target $b$:
\[\st_{b}\B = \left(\bigoplus_{y\in\B_0} \ {}_y\B{_{b}}\right)\ \oplus
\ \left(\bigoplus_{y\in\B_0} \ {}_{b}\B{_y}\right)\]

\end{defi}

\begin{defi}\label{covering}
 Let $\C$ and $\B$ be $k$-categories. A $k$-functor
$F:\C\rightarrow\B$ is a covering of $\B$ if it is surjective on
objects and if $F$ induces $k$-isomorphisms between all
corresponding stars. More precisely, for each $b\in\B_0$ and each
$x$ in the non-empty fibre $F^{-1}(b)$, the map
\[F_{b}^x:\st_x\C\longrightarrow\st_{b}\B.\]
induced by $F$ is a  $k$-isomorphism.

\end{defi}

Observe that a covering is a faithful functor. Note also that Definition \ref{covering} coincides with
the one given
by K. Bongartz and P. Gabriel in \cite{boga}.

\begin{defi} \label{deficover}
Given $k$-categories $\B, \C, \D$, the set of morphisms $\mor(F,G)$ from a covering
$F:\C\rightarrow\B$ to a
covering $G:\D\rightarrow\B$ is the set of
pairs of $k$-linear functors $(H,J)$
where $H: \C  \to \D$, $J: \B \to \B$ are such that $J$ is an isomorphism, $J$ is the identity on objects and $GH=JF$. We also say that $H$ is a $J$-morphism.

We will consider within the group of automorphisms of a covering $F:\C \rightarrow \B$, the subgroup
$\aut_1 F$ of invertible endofunctors $G$ of $\C$ such that $FG=F$.
\end{defi}

For any covering $F$ it is known that $\aut_1 F$  acts freely on each fibre, see \cite{le2, CRS}.

A $k$-category $\B$ is \textbf{connected} if any two objects can be joined by a non-zero walk, see \cite[Section 2]{CRS2} for details.
\begin{defi}
A covering $F: \C\longrightarrow\B$ of $k$-categories is a
\textbf{Galois covering} if $\C$ is connected and $\aut_1 F$ acts
transitively on some fibre.
\end{defi}

As expected in similar Galois theories, the automorphism group acts
transitively at every fibre as soon as it acts transitively on a
particular one, see \cite{le2,CRS}.

Recall that a grading $Z$ of a $k$-category $\B$ by a group $\Gamma_Z$ is a direct sum decomposition of each $k$-module of
morphisms from $b$ to $c$
\[{}_{c} \B_b =\bigoplus_{s\in \Gamma_Z} Z^s{}_{c} \B_b\]
such that for $s,t  \in \Gamma_X$
\[Z^t{}_{d} \B_{c}\    Z^s{}_{c} \B_b \subset  Z^{ts}{}_{d} \B_b.\]

A morphism from $b$ to $c$ is called homogeneous of degree $s$  if it belongs to  $Z^s{}_{c}\B_b$.
A grading is connected if given any two objects they can be joined by a non-zero homogeneous walk of arbitrary degree.
For precise definitions see \cite[Section 2]{CRS2}.

\begin{defi}\cite{CM}
Let $\B$ be a $k$-category and let $Z$ be a grading of $\B$. The smash product category $\B\#Z$ has set of objects
$\B_0\times \Gamma_Z$,
the vector spaces of morphisms are homogeneous components as follows:
\[{}_{(c,t)}\!\left(\B\#Z\right)_{(b,s)}= Z^{t^{-1}s}\ {}_{c} \B_b.\]
\end{defi}

Note that for a connected grading $X$, the evident functor $F_X :  \B\#X\longrightarrow \B$ is a Galois covering.

Let $\B$ be a connected $k$-category with a fixed object $b_0$.
The category $\gal(\B,b_0)$ has as objects the Galois coverings of $\B$.
A morphism in $\gal(\B,b_0)$ from $F:\C\rightarrow\B$ to $G:\D\rightarrow\B$ is a morphism of coverings $(H,J)$,
see Definition \ref{deficover}.

Since any Galois covering $F$ of $\B$ is isomorphic to a smash product Galois covering by considering the natural grading
of $\B$ by $\aut_1F$, we consider as in  \cite{CRS2}  the full subcategory $\gal^\#(\B,b_0)$ whose objects are the smash
product Galois coverings provided by connected gradings of $\B$.
It can be proved that this full subcategory is equivalent to $\gal(\B,b_0)$, see  \cite{CRS2}.
The fibre functor $\Phi^\# : \gal^\#(\B,b_0) \rightarrow \mathsf{Groups}$ given by
\[\Phi^\#(F_X)=F_X^{-1}(b_0)=\Gamma_X\]
is the main ingredient for the definition of the fundamental group, namely
\[\Pi_1(\B,b_0)=\aut\Phi^\#.\]

Next we will consider in detail this group, and we will prove that an element of $\Pi_1(\B,b_0)$ is determined by a family of elements
belonging to the groups of connected gradings related through canonical surjective morphisms $\mu : \Gamma_X \to \Gamma_{X'}$ obtained as soon as $X$ and $X'$ are connected gradings admitting some morphism from $\B\# X$ to $\B\# X'$.

\begin{rem}\label{identifica} For any smash product Galois covering $F_X$ we identify the isomorphic groups
$\aut_1F_X$ and $\Gamma_X$ through left multiplication, that is, by the correspondence $s: F_X \to F_X$
with $s(x)= sx$ for any $s \in \Gamma_X$.
\end{rem}

\begin{pro}
Let $X$ and $X'$ be connected gradings of a $k$-category $\B$, and let $F_X$ and $F_{X'}$ be the corresponding smash product Galois
coverings with groups $\Gamma_X$ and $\Gamma_{X'}$.
Let $(H,J)$ be a morphism from $F_X$ to $F_{X'}$ in $\gal^\#(\B,b_0)$, where $H:\B\#X\longrightarrow \B\# X'$ is given on objects
by $H(b,s)=(b, H_b(s))$. Then there exists a unique surjective morphism of groups
$\lambda_H:\Gamma_X \rightarrow \Gamma_{X'}$ verifying
\[H_b(sx)=\lambda_H(s)H_b(x) \mbox{ for all } x\in\Gamma_X.\]
Moreover the complete list of $J$-morphisms from $F_X$ to $F_Y$ is given by $\left\{qH\right\}_{q\in\Gamma_Y}$,
and \[\lambda_{qH}=q\left(\lambda_H\right)q^{-1}.\]
\end{pro}
For the proof see \cite[Section 2]{CRS2}.

\begin{defi}
In case of existence of a $J$-morphism $H: \B\# X \to \B\# X'$ we consider the \textbf{normalized} $J$-morphism $N=H_{b_0}(1)^{-1} H$. Observe that $N$ does not depend on $H$ since its value is the same when $H$ is replaced by $qH$. Moreover $N(b_0,1) = (b_0,1)$.

We set $\mu=\lambda_N$ and we call $\mu$ the \textbf{canonical group map} associated to the existence of a $J$-morphism from the smash product with $X$ to the one with $X'$. For $H: \B\# X \to \B\# X'$ a morphism we have $$\mu = H_{b_0}(1)^{-1} \lambda_H H_{b_0}(1).$$
\end{defi}
We are now able to describe the elements of the intrinsic fundamental group, namely the automorphisms of the fibre functor.

Recall that $\sigma\in\aut\Phi^\#$ is a family $\left\{\sigma_X:\Gamma_X\to\Gamma_X\right\}$ where $X$ is any connected grading
of $\B$ making commutative the following diagram for any morphism $H$ in $\gal^\#(\B,b_0)$:

\[ \xymatrix{
\Gamma_X  \ar[r]^{\sigma_X} \ar[d]_{H_{b_0}} &   \Gamma_X \ar[d]^{H_{b_0}} \\
\Gamma_{{X'}} \ar[r]^{\sigma_{{X'}}}  &   \Gamma_{{X'}}
} \]

\begin{lem}
The map $\sigma_X$ is the right product by an element $g_X\in\Gamma_X$.
\end{lem}
\begin{proof}
In case $X={{X'}}$, the vertical arrows in the diagram can be specialized by any element in $\aut_1F_X=\Gamma_X$. By Remark
\ref{identifica}, this vertical morphisms are left product
 by some $g\in\Gamma_X$. We infer $\sigma_X(g)=\sigma_X(g1)=g\sigma_X(1)$ and we set $g_X=\sigma_X(1)$.\qed
\end{proof}

\begin{pro}\label{elements}
The automorphisms in $\Pi_1(\B,b_0)$ are in one to one correspondence with families of group elements  $\{g_X\}$ verifying $\mu(g_X)=g_{X'}$ for each canonical group map  $\mu$ corresponding to the existence of a morphism from $\B\#X$ to $B\#X'$.
\end{pro}

\begin{proof}
Let $H$ be a $J$-morphism from $F_X$ to $F_X'$ where $X$ and $X'$ are connected gradings of $\B$, and let
$\mu:\Gamma_X\to\Gamma_{X'}$ be
the corresponding canonical surjective group map. The previous diagram becomes

\[ \xymatrix{
\Gamma_X  \ar[r]^{.g_X} \ar[d]_{H_{b_0}(1)\mu} &   \Gamma_X \ar[d]^{H_{b_0}(1)\mu} \\
\Gamma_{Y'} \ar[r]^{.g_{Y'}}  &   \Gamma_{Y'}
} \]
Hence \[H_{b_0}(1)\mu(gg_X)=H_{b_0}(1)\mu(g)g_{Y'}\] for any $g$. Since $\mu$ is a group homomorphism we infer $\mu(g_X)=g_{Y'}$.
Reciprocally a family of elements with the stated property clearly defines an automorphism of the fibre functor.\qed
\end{proof}

We say that a family of group elements is \textbf{compatible} if it satisfies the condition in the proposition above.

\section{\sf Functoriality and incompressible subcategories}

An important tool for proving the functoriality of the intrinsic fundamental group is the restriction of a grading to full
subcategories that we will consider below.

Let $Z$ be a non necessarily connected grading of a connected $k$-category $\B$. Let ${}_{b_2}\!(\Gamma_Z)_{b_1}$ be the set of
\textbf{walk's degrees} from $b_1$ to $b_2$, that is, the set of elements in $\Gamma_Z$ which are degrees of homogeneous
non-zero walks from $b_1$ to $b_2$. Note that if $b_1=b_2$ this set is a subgroup of $\Gamma_Z$ which we denote $\Gamma_{Z,b_1}$.
Let $s$ be any walk's degree from $b_1$ to $b_2$. Then
\[{}_{b_2}\!(\Gamma_Z)_{b_1}=s\left[\Gamma_{Z,b_1}\right]=\left[\Gamma_{Z,b_2}\right]s.\]

Recall that by definition, a grading $X$ is \textbf{connected} if and only if for any objects $b_1, b_2 \in \B_0$
\[{}_{b_2}\!(\Gamma_X)_{b_1}=\Gamma_X.\]
This is equivalent to ${}_{b_2}\!(\Gamma_X)_{b_1}=\Gamma_X$ for some pair of objects, and in particular to $\Gamma_{X,b_0} = \Gamma_X$ for a given object $b_0$, see \cite{CRS2}.

Let $\B$ be a $k$-category and let $X$ be a connected grading of $\B$ by the group $\Gamma_X$. Let $\D$ be a connected full
subcategory of $\B$. The \textbf{restricted grading} by the same group $\Gamma_X$ is denoted $X\!\downarrow_\D$.
Note that since $\D$ is full, each $k$-module of morphisms in $\D$ has the same direct sum decomposition from the original
grading. Clearly the grading $X\!\downarrow_\D$ is not connected in general.

In order to consider the corresponding connected grading out of a non connected one we will provide some preliminary results. The description of a connected component will make use of conjugated gradings as defined in \cite{CRS2}. More precisely, let $Z$ be a non necessarily connected grading of a connected $k$-category $\B$ and let $(a_b)_{b\in \B_0}$ be a family of elements in $\Gamma_Z$.  The \textbf{conjugated grading} ${}^a\! Z$ of $\B$ by the same group $\Gamma_Z$ is as follows:
we set the ${}^a\! Z$-degree of a non-zero homogeneous morphism from $b$ to $c$ of $Z$-degree $t$ to be $a_c^{-1}t a_b$, that is,
\[({}^a\! Z)^s {}_c\B_b = Z^{a_c s a_b^{-1}}  {}_c\B_b.\]
Note that the underlying homogeneous components remain unchanged, and there is no difficulty for proving that ${}^a\! Z$ is indeed a grading. We observe that ${}^a\! Z$ is connected if $Z$ is so.

\begin{defi}
Let $\B$ be a connected $k$-category with a non necessarily connected grading $Z$ with group $\Gamma_Z$. Let $b_0$ be a fixed
object. The connected component of $\B\# Z$ containing the object $(b_0,1)$ is denoted $(\B\# X)^0$ and is called the \textbf{connected component of the base object}.
\end{defi}

\begin{pro}\label{outconnected}
The connected component of the base object $(\B\# Z)^0$ is the smash product of $\B$ by a conjugated grading of $Z$. The group of this Galois covering (or of the corresponding connected grading) is $\Gamma_{Z,b_0}$.
\end{pro}

\begin{proof}
Firstly we assert that the restriction of $F_Z : \B\# Z \to \B$ to the connected component of the base object is still surjective on objects. Indeed  $\B$ is connected, then for any object $b\in\B_0$ there exists a walk $w$ from $b$ to $b_0$. For each morphism appearing in $w$ we choose an homogeneous non-zero component, obtaining in this way an homogeneous walk $w'$ of some degree $s$ from $b$ to $b_0$. Lifting appropriately the homogeneous morphisms of $w'$ provides a walk in the connected component of the base object from some $(b,s)$ to $(b_0,1)$. Thus $(b,s)$ is an object of the connected component of the base object lying in the fibre of $b$ of the restriction of $F_Z$. Consequently this restriction is a Galois covering.

Secondly we know  that each Galois covering  provides a connected grading  of the base category (with group the automorphism group of the covering) once a family of objects is chosen in each fibre (see \cite{CM}). Let $(b,u_b)$ be such a choice for the restriction of $F_Z$, with $u_{b_0}=1$. Since $(\B\# X)^0$ is connected, there is a walk from $(b_0,1)$ to $(b,u_b)$ providing through $F_Z$ an homogeneous walk in $\B$ from $b$ to $b_0$ of $Z$-degree $u_b$. We set $u=(u_b)_{b\in\B_0}$. The definition of the grading (see \cite{CM}) shows that the grading coincides with ${}^uX$.

Finally in order to show that the group of the grading (namely the automorphism group of the Galois covering) is $\Gamma_{Z,b_0}$ consider $(b_0, s)$ an object in the fibre of $b_0$ and a walk from $(b_0,s)$ to $(b_0, 1)$. Its image using $F_Z$ provides an homogeneous closed walk at $b_0$ in $\B$ which degree is precisely $s$ according to the construction of the grading. Conversely let $w$ be an homogeneous closed walk at $b_0$ in $\B$. By appropriate lifting, $w$ provides a walk in $\B\# Z$ between $(b_0,1)$ and some $(b_0, s)$, where the degree of $w$ is $s$.\qed

\end{proof}

\begin{rem}\label{choices}
The connected grading obtained depends on the choice $(u_b)_{b\in\B_0}$, where
$u_b\in{}_{b}(\Gamma_Z)_{b_0}$. Any other choice $(u'_b)_{b\in\B_0}$ with $u'_{b_0}=1$ is obtained as
$(u_ba_b)_{b\in\B_0}$ where $a_b\in\Gamma_{Z,b_0}$ and $a_{b_0}=1$. The group of the connected gradings remains the same.
\end{rem}

\begin{thm}
Let $\B$ be a connected $k$-category with a fixed object $b_0$ and let $\D$ be a connected full subcategory containing $b_0$.
Then there is a canonical group morphism
\[\kappa : \Pi_1(\D, b_0)\longrightarrow \Pi_1(\B,b_0)\]

In this way $\Pi_1$ becomes a functor from the category of small $k$-categories with a chosen base object with morphisms the fully faithful functors which are injective on objects and  preserve base objects to the category of groups.
\end{thm}

\begin{proof}
According to Proposition \ref{elements} let $\sigma\in\Pi_1(\D, b_0)$ be determined by  a compatible family $\{g_Y\}$ where $Y$ varies over all the connected gradings of $\D$ and where $g_Y$ is in $\Gamma_Y$. Recall that $\mu(g_Y)=g_{Y'}$ for each pair of connected gradings $Y$ and $Y'$ such that there is a morphism from $\D\#Y$ to $\D\#Y'$.
In order to define $\left(\kappa(\sigma)\right)_X$ for a given connected grading $X$ of $\B$ we consider the connected component of the base object $\left(\D\#X\!\!\downarrow_D\right)^0$. According to the previous proposition this  Galois covering is given by a smash product with respect to a connected grading ${}^u\!\left(X\!\downarrow_D\right)$  with group $\Gamma_{X\downarrow_D,b_0}$. We define
\[\kappa(\sigma)_X=g_{\ {}^u\!\left(X\!\downarrow_D\right)}.\]

In order to check that $\kappa$ is  well defined we have to verify that for any set $u'=(u'_b)$ of degrees of homogeneous walks in $\D$ from $b$ to $b_0$ with $u'_{b_0}=1$ we have   $$g_{\ {}^{u}\left(X\!\downarrow_\D\right)}=g_{\ {}^{u'}\left(X\!\downarrow_\D\right)}.$$
This will be insured by the following Lemma, which shows that there is a morphism between the corresponding Galois smash
coverings whose corresponding canonical group map $\mu$ is the identity.

We prove now that the obtained family is compatible. Let $X$ and $X'$ be two connected gradings of $\B$,
let $(H,J)$ be a morphism of coverings $\B\#X\to\B\#X'$ and let $\mu :\Gamma_X\to\Gamma_{X'}$ be the corresponding canonical group map.

Since $J$ is the identity on objects it restricts to an isomorphism of $D$. Hence $(H,J)$ restricts to a morphism from the full subcategory $\left(\D\#X\!\!\downarrow_D\right)$ to the full subcategory $\left(\D\#X'\!\!\downarrow_D\right)$.

Note that a morphism between non necessarily connected coverings is faithful, consequently it preserves connected components. Hence the preceding restriction gives a morphism of Galois coverings $$\left(\D\#X\!\!\downarrow_D\right)^0 \to \left(\D\#X'\!\!\downarrow_D\right)^0.$$
The canonical group map arising from this morphism is the restriction of the canonical $\mu$ associated to $(H,J)$. This shows that the family defined by $\kappa (g)$ is compatible.

As a consequence we note that the image of the restriction of $\mu$ to  $\Gamma_{X\downarrow_D,b_0}$ is in  $\Gamma_{X'\downarrow_D,b_0}$, a fact which can also be obtained easily directly.

Finally note that $\kappa$ clearly preserves composition of inclusions of full subcategories. \qed
\end{proof}

\begin{rem}
The connected gradings ${}^{u}\!\left(X\!\downarrow_\D\right)$ and ${}^{u'}\!\left(X\!\downarrow_\D\right)$ are conjugated gradings by the family $(a_b)_{b \in \B_0}$  with $a_{b_0}=1$.
\end{rem}

\begin{lem}
Let $X$ be a connected grading of $\B$ and let $(a_b)_{b \in \B_0}$ be a family of elements in $\Gamma_X$.
There is a covering morphism $\B\#X\to\B\# {}^a\!X$ between conjugated gradings. The corresponding induced canonical
group morphism $\mu:\Gamma_X \to\Gamma_X$ is conjugation by $a_{b_0}$. In particular if $a_{b_0}=1$ then $\mu=1$.
\end{lem}

\begin{proof}
The functor $H$ is given on objects by $H(b,s)=(b,sa_b)$ while on morphisms the functor is the identity since
\[{}_{(c,ta_c)}(\B\#{}^a\!X)_{(b,sa_b)}=({}^a\!X)^{a_c^{-1}t^{-1}sa_b} {}_c\B_b = X^{t^{-1}s}{}_b\B_c = {}_{(c,t)}\left(\B\#X\right)_{(b,s)}.\]

In order to compute $\mu$ we first normalize $H$ by considering $N=H_{b_0}(1)^{-1} \ H$. Since $H_{b_0}(1)=a_{b_0}$ we infer $N(s)=a_{b_0}^{-1}sa_{b_0}.$\qed
\end{proof}

We end this section with a general criterion for $\kappa$ being injective, and we give a family of cases where the
criterion applies.

\begin{cor}
Let $\B$ be a connected $k$-category, let $b_0$ be a fixed object and let $\D$ be a connected full subcategory containing $b_0$.
Assume any connected grading of $\D$ is of the form ${}^u\left(X\!\downarrow_\D\right)$ for some connected grading $X$ of $\B$.
Then the group morphism
\[\kappa : \Pi_1(\D, b_0)\longrightarrow \Pi_1(\B,b_0)\]
is injective.
\end{cor}
\begin{proof}
Let $\sigma=(g_Y)$ be a compatible family defining an element in $ \Pi_1(\D, b_0)$. Assume $\kappa(\sigma)=1$ which
means that for any connected grading $X$ of $\B$ we have $\kappa(\sigma)_X=1$. Recall that
$\kappa(\sigma)_X= g_{\left[{}^u\left(X\!\downarrow_\D\right)\right]}$. Consequently those elements are trivial.
By hypothesis any connected grading $Y$ of $\D$ is of this form, then $\sigma=1$.\qed
\end{proof}

\begin{defi}
A subcategory $\D$ of $\B$ is said to be \textbf{convex} if any morphism of $\D$ only factors through morphisms in $\D$. In case $\D$ is full, this condition is equivalent to the fact that any composition of an outcoming morphism (with source in $\D$ and target not in $\D$) and an incoming one (reverse conditions) must be zero.
\end{defi}

\begin{cor}
Let $\B$ be a connected $k$-category, let $b_0$ be a fixed object and let $\D$ be a connected full convex subcategory containing $b_0$. Then $\kappa$ is injective.
\end{cor}
\begin{proof}
Let $Y$ be a connected grading of $\D$. We extend $Y$ to $\B$ by providing trivial degree to any morphism whose source or target is not in $\D$.  By hypothesis there is no non-zero morphism of the form $gf$ where $f$ has source in $\D$,  $g$ has target in $\D$, and the source of $g$ and the target of $f$ coincide without being in $\D$. We infer that this setting indeed provides a grading. The grading is connected since any element of the group is a walk's degree, already in $\D$. The preceding result insures that $\kappa$ is injective.  \qed
\end{proof}


\footnotesize
\noindent C.C.: Universit\'e Montpellier 2,
\\Institut de math\'{e}matiques et de mod\'{e}lisation de Montpellier I3M,\\
UMR 5149\\
Universit\'{e}  Montpellier 2,
\\F-34095 Montpellier cedex 5,\\
{\tt Claude.Cibils@math.univ-montp2.fr}

\medskip

\noindent M.J.R.: Universidad Nacional del Sur,
\\Departamento de Matem\'atica,
Universidad Nacional del Sur,\\Av. Alem 1253\\8000 Bah\'\i a Blanca,
Argentina.\\ {\tt mredondo@criba.edu.ar}

\medskip

\noindent A.S.:
\\Departamento de Matem\'atica,
 Facultad de Ciencias Exactas y Naturales,\\
 Instituto de Matem\'atica Luis Santal\'o, IMAS-CONICET\\
 Universidad de Buenos Aires,
\\Ciudad Universitaria, Pabell\'on 1\\
1428, Buenos Aires, Argentina. \\{\tt asolotar@dm.uba.ar}

\end{document}